\documentclass[12pt]{amsart} 

\usepackage[all,cmtip,2cell]{xy}

\usepackage{enumitem}
\usepackage{amsmath,amsthm,amssymb}
\usepackage{words}
\usepackage{geometry}
\usepackage[colorlinks=true]{hyperref}
\usepackage{widebar}
\usepackage[dvipsnames]{xcolor}

\usepackage{comment}

\providecommand{\Log}{\cat{Log}}

\providecommand{\ocM}{{\woverbarchoice{\mathcal}{M}}}
\providecommand{\cE}{{\mathcal{E}}}
\providecommand{\cW}{{\mathcal{W}}}
\providecommand{\ucX}{{\wunderbarchoice{\mathcal}{X}}}

\providecommand{\uD}{{\wunderbarchoice{\mathnormal}{D}}}

\providecommand{\cI}{{\mathcal{I}}}

\providecommand{\uX}{{\wunderbarchoice{\mathnormal}{X}}}
\providecommand{\uY}{{\wunderbarchoice{\mathnormal}{Y}}}
\providecommand{\uZ}{{\wunderbarchoice{\mathnormal}{Z}}}
\providecommand{\uU}{{\wunderbarchoice{\mathnormal}{U}}}
\providecommand{\uV}{{\wunderbarchoice{\mathnormal}{V}}}
\providecommand{\oCp}{{\oC\vphantom{C}'}}

\providecommand{\cA}{\mathcal{A}}
\providecommand{\fpr}{\mathop{\times}}

\mathchardef\ordinarycolon\mathcode`\:
\mathcode`\:=\string"8000
\begingroup \catcode`\:=\active
  \gdef:{\mathrel{\mathop\ordinarycolon}}
\endgroup
\newcommand\displaceamount{2.8pt}

\newcommand{\doubledown}{\ar@<\displaceamount>[d]\ar@<-\displaceamount>[d]}

\newcommand{\doubleright}{\ar@<\displaceamount>[r]\ar@<-\displaceamount>[r]}

\newcommand{\change}[1]{#1}

\newtheoremstyle{named}%
	{}%
	{}%
	{\itshape}%
	{}%
	{\bfseries}%
	{.}%
	{.5em}%
	{\thmname{#1 #3}}

\newtheorem{theorem}{Theorem}
\newtheorem{proposition}[theorem]{Proposition}
\newtheorem{corollary}[theorem]{Corollary}

\newtheorem{lemma}[theorem]{Lemma}

\theoremstyle{remark}
\newtheorem{remark}[theorem]{Remark}

\newtheorem{example}[theorem]{Example}
\theoremstyle{named}

\theoremstyle{definition}
\newtheorem{definition}[theorem]{Definition}

\numberwithin{theorem}{subsection}

\begin{document}

\title{Birational invariance in logarithmic Gromov--Witten theory}

\author{Dan Abramovich}

\author{Jonathan Wise}
\date{\today}

\address[Abramovich]{Department of Mathematics\\
Brown University\\
Box 1917\\
Providence, RI 02912\\
U.S.A.}
\email{abrmovic@math.brown.edu}

\address[Wise]{University of Colorado, Boulder\\
Boulder, Colorado 80309-0395\\ USA}

\email{jonathan.wise@math.colorado.edu}

\thanks{Abramovich supported in part by NSF grant
   DMS-1162367. Wise is supported by an NSA Young Investigator's Grant, Award \#H98230-14-1-0107.}
\begin{abstract} Gromov--Witten invariants have been constructed to be deformation invariant, but their behavior under other transformations is subtle.  We show that {\em logarithmic} Gromov--Witten invariants are also invariant under appropriately defined {\em logarithmic modifications}. 
\end{abstract} 
\maketitle

\setcounter{tocdepth}{1}
\setcounter{secnumdepth}{4}
\tableofcontents

\section{Introduction}\label{sec:intro}

\subsection{Main result}

\change{In this paper we answer the following question, posed by Mark Gross.} Consider two logarithmically smooth complex projective varieties $X$ and $Y$ and a logarithmic modification $h: Y \to X$ between them.  \change{How are the logarithmic Gromov--Witten invariants of $X$ and $Y$ related?}  We show in Theorem \ref{maintheorem} that the canonical morphism $\ocM(Y) \to \ocM(X)$ between the associated spaces of logarithmic stale maps is virtually birational, and, as a consequence, the Gromov--Witten theories {with primary insertions coming from $X$} coincide (see Corollary \ref{Cor:invariance}).

If $X$ is a proper logarithmic scheme, there is a logarithmic algebraic stack $\ocM(X)$ parameterizing stable logarithmic maps from logarithmic curves into $X$ \cite{GS, Chen, AC, minimal}.  When $X$ is also logarithmically smooth, the underlying algebraic stack of $\ocM(X)$ has a virtual fundamental class $[ \ocM(X) ]^{\rm vir}$ that can be used to define Gromov--Witten invariants.

A  \emph{logarithmic modification}  is a  proper, birational, logarithmically \'etale morphism $Y \to X$.  By \cite[Theorem B.6]{AMW}, a logarithmic modification $Y \rightarrow X$ induces a morphism $\ocM(Y) \rightarrow \ocM(X)$.

\begin{theorem}\label{maintheorem} 
	Let $h : Y \rightarrow X$ be a logarithmic modification of logarithmically smooth schemes inducing a projection $\pi : \ocM(Y) \rightarrow \ocM(X)$.  Then
	\begin{equation*}
		\pi_* \left( [\ocM(Y)]^\vir\right) =  [\ocM(X)]^\vir .
	\end{equation*}
\end{theorem}

We will work throughout this paper in the language of logarithmic schemes. Following Ogus, we refer to the underlying scheme of a logarithmic scheme $X$ by decoration with an underline:  $\uX$.  

\subsection{Toroidal structures and logarithmic structures}

	We summarize the relationship between toroidal embeddings and logarithmic structures, giving a brief summary of the basic definitions.  For an authentic introduction we refer the reader to \cite{Kato} or to \cite{handbook}.

 Logarithmic structures are a recent addition to algebraic geometry and Gromov--Witten theory, but logarithmically smooth varieties and logarithmic modifications between them have a concrete classical description in terms of toroidal embeddings and toroidal modifications.

\change{A \emph{toroidal embedding} is an open subset $\uU \subset \uX$ such that, if {$\widehat \uX=\Spec \widehat \cO_{X,x}$}
is the formal completion of $\uX$ at any {closed point $x$}, and $\widehat \uU$ is the preimage of $\uU$ in $\widehat \uX$, then the pair $(\widehat X, \widehat \uU)$ is isomorphic to $(\widehat V, \widehat T)$, where $\widehat V$ and $\widehat T$ are constructed in the same way {at a point $v$ of} a toric variety $V$ with dense torus orbit $T$; in other words, there is an isomorphism $\widehat \cO_{X,x}\to \widehat \cO_{V,v}$ carrying $\cI_{X\smallsetminus U}\widehat \cO_{X,x}$ to $\cI_{V\smallsetminus T}\widehat \cO_{V,v}$ as in~\cite[Definition~II.1.1]{KKMS}.  The toric variety $V$ may depend on the choice of point in $X$ where we perform the completion.}
	\change{Informally,} a toroidal variety $\uX$ locally looks formally, and therefore also \'etale locally, like a toric variety.  Similarly, a dominant morphism $\uY \to \uX$ of varieties with toroidal structures $\uU_Y \subset \uY, \uU_X \subset \uX$ is {\em toroidal} if  it locally looks like a torus equivariant morphism of toric varieties.  This notion of \change{toroidal} morphisms  was introduced in \cite[Definition 1.3]{AK}, but the birational case was already present in \cite[Definition~II.2.1, Theorem~II.2.1$^\ast$, Definition~II.2.3]{KKMS}.  We note that sometimes one indicates the toroidal structure on $\uX$ by specifying the divisor $\uD_X := \uX \setminus \uU$ instead of $\uU$.

	A \emph{logarithmic structure} on a scheme $\uX$ is an \'etale sheaf of monoids $M_X$ and a homomorphism $\varepsilon : M_X \rightarrow \change{\cO_\uX}$, \change{with $\cO_\uX$} given its multiplicative monoidal structure, such that every unit of \change{$\cO_\uX$} is the image of a unique element of $M_X$.  The triple $X = (\uX, M_X, \varepsilon)$ is called a \emph{logarithmic scheme}.  
	
	If $i : \uU \subset \uX$ is an open subset then $M_X = \cO_{\uX} \fpr_{i_\ast \cO_{\uU}} i_\ast \cO_{\uU}^\ast$ (the fiber product taken as sheaves in the \'etale topology) along with the projection to $\cO_{\uX}$ gives a logarithmic structure on $\uX$.  Concretely, the local sections of $M_X$ are the sections of $\cO_{\uX}$ that become invertible when restricted to $\uU$.
	When $\uU \subset \uX$ is a toroidal embedding, the logarithmic scheme $X$ just constructed is fine, saturated, and logarithmically smooth and $\uU$ can be recovered as the open subset of $\uX$ on which the \change{map $\varepsilon : M_X \rightarrow \cO_{\uX}$ is an isomorphism onto $\cO_{\uX}^\ast$}.

	A morphism of logarithmic schemes $f : Y \rightarrow X$ consists of a morphism of the underlying schemes $\uY \rightarrow \uX$, for which we use the same symbol $f$, and a morphism of \'etale sheaves of monoids $f^{-1} M_X \rightarrow M_Y$ such that the diagram
	\begin{equation*} \xymatrix{
			f^{-1} M_X \ar[r]^{\varepsilon} \ar[d] & f^{-1} \cO_X \ar[d] \\
			M_Y \ar[r]^{\varepsilon} & \cO_Y
	} \end{equation*}
	commutes.  It can be seen easily from the definition that a morphism of toroidal varieties $(\uY, \uV) \rightarrow (\uX, \uU)$ {(in which $\uV$ maps to $\uU$)} induces a morphism of logarithmic schemes.  Moreover, a {\emph{toroidal}} morphism $(\uY, \uV) \rightarrow (\uX, \uU)$ such that $\uV \rightarrow \uU$ is \'etale or smooth induces a morphism of logarithmic schemes $Y \rightarrow X$ that has the same property, logarithmically, by \cite[Proposition~(3.4)]{Kato}.

\subsection{Implication for logarithmic Gromov--Witten invariants} 

	If $X$ is a projective logarithmic scheme, $\ocM(X)$ decomposes as a disjoint union of open and closed substacks $\ocM_\Gamma(X)$, each of finite type with a projective coarse moduli space, indexed by combinatorial data $\Gamma$.  This was proved with additional technical hypotheses in \cite{GS, AC}, but those restrictions {were eliminated, using some of the methods developed in this paper, in \cite{logbd}.}  Associated to each datum $\Gamma$ one has \emph{logarithmic Gromov--Witten invariants} of $X$, as defined in \cite{GS, AC}.

\begin{corollary}\label{Cor:invariance}
	Let $Y \to X$ be a logarithmic modification of logarithmically smooth schemes. Then the logarithmic Gromov--Witten invariants  of $X$ and $Y$ {with primary insertions coming from $X$} coincide: given numerical data $\Gamma_X$ on $X$ there is a unique choice of numerical data $\Gamma_Y$ with $h_*\Gamma_Y = \Gamma_X$, such that
	$$\langle \alpha_1\cdots \alpha_n\rangle^X_{\Gamma_X} = \langle h^*\alpha_1\cdots h^*\alpha_n\rangle^Y_{\Gamma_Y},$$
	and for all other choices $\Gamma'_Y$ with $h_*\Gamma'_Y = \Gamma_X$ \change{the invariants vanish.}
	
\end{corollary}

We stress that our result applies only for toroidal morphisms. Suppose $X=\PP^2$ with its toric structure. Our result applies when $Y$ is the toric blowing up of $X$ at a torus fixed point, such as the origin. It does {\em not} apply if $Y$ is the non-toric blowing up of $X$ at any point which is not fixed by the torus. 

On the one hand this result is to be expected: consider the case of a toroidal degeneration $\pi:X \to B$, where $B$ is a curve with toroidal divisor  $D_B = \{b_0\}\in B$ and $D_X = \pi^{-1} D_B$.  Suppose given a birational modification $Y \to X$ such that $Y \to B$ is also toroidal. This implies that over a general point $b\neq b_0$ of $B$ we have $Y_b = X_b$, so they have identical Gromov--Witten invariants. This implies that  the most important Gromov--Witten invariants of $X$ and $Y$, namely those with fiberwise curve classes and global insertions, automatically coincide, whether or not the morphism $Y \to X$ is toroidal.

On the other hand this result may be somewhat surprising. 
There are curve classes on $Y$ which are not present on $X$, for instance an exceptional curve $E$.
The corollary says in particular that all logarithmic Gromov--Witten invariants on $Y$ with curve class $dE$  vanish. In fact in this case the moduli spaces $\ocM_{\Gamma'_Y}(Y)$ are {\em empty}, \change{in dramatic} contrast with \change{conventional} Gromov--Witten invariants.

\subsection{Gromov--Witten invariants and birational invariance}
Algebraic Gromov--Witten invariants are virtual curve counts on a complex projective variety $X$, thus are biregular invariants. The formalism of virtual fundamental class shows that they are  automatically deformation invariant: if $X$ appears as a fiber of a smooth family, then its invariants coincide with the invariants of other smooth fibers. This property is fundamental in Gromov--Witten theory.%

By contrast, the behavior of  Gromov--Witten invariants under a birational transformation $Y \to X$ is generally subtle.  Many have studied this subtlety and found that good behavior can be obtained in special situations.  Here is a non-exhaustive list:
\begin{enumerate}
\item Gathmann \cite[Theorem 2.1]{Gathmann} provided a procedure for calculating the behavior of genus-0 invariants under point blowing up.
\item J. Hu \cite[Theorem 1.2]{Hu} showed the birational invariance of genus $\leq 1$ Gromov--Witten numbers under blowing up a point or a smooth curve, as well as  arbitrary genus invariants when $\dim X \leq 3$.
\item Lai \cite[Theorem 1.4]{Lai} showed the birational invariance in genus 0  if $Y \to X$ is the blowing up of a smooth subvariety $Z$ with convex normal bundle with enough sections, or if $Z$  \change{contains no images of $\PP^1$.}
\item Manolache \cite[Proposition 5.14]{Manolache} showed birational invariance in genus 0 if $Z$ is the transversal intersection of $X$ with a smooth subvariety of an ambient homogeneous space.
\end{enumerate}
A number of authors, including  D.\ Maulik and R.\ Pandharipande \cite{MP} and  J. Hu, T.\ J.\ Li and Y.\ Ruan \cite{HLR}, considered the behavior of invariants under blowing up using the degeneration formula.

Theorem \ref{maintheorem}  shows that logarithmic Gromov--Witten invariants are well-suited to questions of birational invariance. {It would be interesting to understand how Gromov--Witten invariants with primary and descendant insertions from $Y$ behave.  It would also} be interesting to obtain comparison mechanisms between logarithmic and usual invariants similar to the results of \cite{MP}. Such a mechanism should allow a comparison of Gromov--Witten invariants of $X$ and $Y$ even if $Y \to X$ is not a toroidal morphism.

\subsection{Artin fans}

Let $\Log$ denote the stack of logarithmic structures introduced in \cite{Olsson_log}.  As explained in op.\ cit., a logarithmic variety $X$ is logarithmically smooth if and only if  the associated map $\uX \to \Log$ is smooth.
As we show in Section~\ref{Sec:toric-stacks} below
this map factors as $\uX \to \ucX \to \Log$ where  $X \to \cX$ is a strict smooth map and $\cX$ is a ``locally toric stack'', meaning is has an \'etale cover by finitely many stacks of the form $[V/T]$, where $V$ is a toric variety and $T$ its dense torus.
The stack $\cX$ is logarithmically \'etale over a point.
We show in Corollary~\ref{cor:log-mod-str} that the map $Y \to X$ is obtained as the pullback of a toric modification $\cY \to \cX$.  In the local picture, this means that $V' = \cY \times_\cX V$ is a toric variety for the same torus $T$.

We intended to name the stack $\cX$ ``the Olsson fan of $X$", however the name ``the Artin fan of $X$" seems to have stuck, and we will use it here.

The construction of $\cX$ has its origin in unpublished notes on gluing Gromov--Witten invariants by Q.\ Chen and by M.\ Gross. Those notes showed that a solid treatment of Artin fans would require a significant amount of pain.  Having endured it, we hope we have managed \change{to hide} this pain and present in this paper a pleasant theory.

A precursor of Artin fans in a special case was given in \cite{ACFW}.  A further generalization of our treatment here, allowing arbitrarily singular logarithmic schemes, is given in \cite[Section~3.1]{logbd}). Since then Artin fans have taken a life of their own: Ulirsch \cite{Ulirsch_Artinfanstrop} shows that the Berkovich analytification $\cX^\beth$ of the Artin fan of $X$  provides an analytic structure on the tropicalization of $X$. Ranganathan \cite{Ranganathan-super} shows that superabundance of tropical geometry is explained by obstructions to lifting curves from $\cX$ to $X$. In \cite{Ranganathan-toric} Ranganathan uses Artin fans as a tool in giving a toroidal description of the space of logarithmic stable maps of genus 0 in the toric case.

\subsection{Outline of proof}

The structure of the proof is very similar to that of the main theorem in \cite{AMW}.

In Section \ref{Sec:mcx+mcy} we construct moduli stacks of pre-stable logarithmic maps $\fM(\cY)$ and $\fM(\cX)$ with maps $\psi_X:\ocM(X) \to \fM(\cX)$ and $\psi_Y:\ocM(Y) \to \fM(\cY)$ constucted in Section \ref{Sec:cartesian}. We show
\begin{proposition}[\change{see Propositions~\ref{prop:log-etale-over-point} and~\ref{prop:rel-alg-stack}}] \label{Prop:mcx-log-smooth}
The stacks $\fM(\cY)$ and $\fM(\cX)$ are algebraic and \change{are} logarithmically smooth.
\end{proposition}

In order to compare the moduli spaces we construct another stack $\fM'(\cY\to \cX)$ in Section~\ref{Sec:M'}, as well as morphisms  $\psi'_Y:\ocM(Y) \to \fM'(\cY\to \cX)$ and $\alpha: \fM'(\cY\to\cX)\to \fM(\cY)$ in Section~\ref{Sec:cartesian}, such that $\psi_Y = \alpha\circ \psi'_Y$.  We show

\begin{proposition}[see Corollary~\ref{Cor:M'-algebraic}, Lemma~\ref{lem:strictness}, and Section~\ref{Sec:M'toM}]\label{Prop:M'} The  stack $\fM'(\cY\to \cX)$ is algebraic and the morphism $\alpha$ is \'etale and strict.
\end{proposition}

 We construct $\fM'(\cY\to \cX)$ with a morphism $\fM(h): \fM'(\cY\to \cX) \to \fM(\cX)$. We obtain a diagram
\begin{equation} \label{Eq:Costello}
\vcenter{\xymatrix{
\ocM(Y) \ar[rr]^-{\ocM(h)}\ar[d]_{\psi'_Y} && \ocM(X)\ar[d]^{\psi_X}\\
\fM'(\cY \to \cX) \ar[rr]_-{\fM(h)}&&  \fM(\cX)
}}
\end{equation}
 and prove
\begin{proposition}[see Section \ref{Sec:cartesian}]\label{Prop:cartesian}
Diagram~\eqref{Eq:Costello} is cartesian.
\end{proposition}

\begin{proposition}[see Proposition \ref{prop:univ}]\label{Prop:Costello}
The morphism $\fM(h)$ is of pure degree 1.
\end{proposition}

We construct obstruction theories $\cE_X$ relative to $\psi_X$ and $\cE_Y$ relative to $\psi_Y$ and prove
\begin{proposition}[see Proposition \ref{prop:obs}]\label{Prop:relative-obstruction}
 We have $$[\ocM(X)]^\vir = (\psi_X)_{\cE_X}^! [\change{\fM(\cX)}],\text{\qquad  \qquad} [\ocM(Y)]^\vir =   \change{(\psi'_Y)_{\cE_Y}^! [\fM'(\cY\to \cX)],} $$ 
 \change{and $\ocM(h)^\ast\cE_X = \cE_Y$. }
\end{proposition}

Theorem \ref{maintheorem} then follows from  Costello's result \cite[Theorem 5.0.1]{Costello}; see also \cite[Proposition 5.29]{Manolache} and \cite[Proposition 3.15]{Lai}.

\subsection{Conventions}

We work over \change{an algebraically closed field $k$} of characteristic zero.  With one exception (in Proposition~\ref{prop:proper-base-change}) all logarithmic structures in this paper are fine and saturated.%

\subsection{Acknowledgements} \change{We thank Mark Gross who asked the question and offered several suggestions, and Steffen Marcus, with whom some of the techniques used here were developed in~\cite{AMW}. Special thanks are due to Jim Bryan, who pointed out that understanding primary and descendant insertions from $Y$ would require additional work.  We are very grateful as well to Zhi Jin, Martin Ulirsch, and the anonymous referee for their careful reading, comments, questions, and suggestions.}

\section{Construction of $\cX$ and $\cY$}\label{Sec:toric-stacks}

We construct Artin fans only for logarithmically smooth logarithmic schemes.  A more general construction appears in \cite{logbd}.  The general case is also treated in \cite{Ulirsch}, where it is connected to Kato fans, polyhedral complexes and Berkovich analytic spaces.

We construct the stack $\cX$ as a universal object depending on $X$.  First, there is a canonical morphism $X \to \Log$; its image is an open substack of $\Log$, but it is too coarse an object because different strata of $X$ can map to the same point of $\Log$.  The idea is to correct this deficiency in a universal way.  We then construct $\cY$ by repeating the same construction, this time working relative to $\Log(\cX)$.

\subsection{Connected components of the fibers of a smooth morphism}

Let $f : X \rightarrow Y$ be a smooth, quasicompact morphism of schemes.  Let $\pi_0(X/Y)$ be the \'etale $Y$-space defined in \cite[Section~(6.8)]{LMB}.  A point of $\pi_0(X/Y)$ lying above a geometric point $y$ of $Y$ corresponds to a connected component of the fiber $X_y$.  We will generalize the construction of $\pi_0(X/Y)$ to a smooth, quasicompact morphism of algebraic stacks.  

\begin{remark}
This section is closely related to, and overlaps somewhat with, \cite[Section~4.1]{minimal}.  As in loc.\ cit., one could eliminate the smoothness requirement and replace it with local finite presentation, flatness, and reduced geometric fibers, but that generality is not necessary here.
\end{remark}

\begin{proposition} \label{prop:pi_0-univ-prop}
\change{Let $X \rightarrow Y$ be a smooth morphism of schemes.  Then} $\pi_0(X/Y)$ is the initial factorization of $X \rightarrow Y$ through an \'etale $Y$-space.
\end{proposition}
\begin{proof}
It will be sufficient to show that for any other such factorization $X \rightarrow Z \rightarrow Y$, there is an inclusion $X \fpr_{\pi_0(X/Y)} X \subset X \fpr_Z X$ as open subschemes of $X \fpr_Y X$.  For this it is sufficient to show there is an inclusion on the level of points.  Since everything in sight commutes with base change in $Y$, we may assume $Y$ is the spectrum of a separably closed field.  In this case, the inclusion reduces to the well-known universal property of $\pi_0(X) = \pi_0(X/Y)$.
\end{proof}

As the formation of $\pi_0(X/Y)$ commutes with base change in $Y$, the definition extends to a smooth, quasicompact morphism $f : X \rightarrow Y$ that is representable by schemes.  We show it can be extended to an arbitrary morphism of algebraic stacks.

\change{First,} let $f : X \rightarrow Y$ be a smooth morphism from an algebraic stack to a \change{\emph{scheme} $Y$}.  Regard $\Phi : X' \mapsto \pi_0(X'/Y)$ as a functor from the category of smooth $X$-schemes to the category of \'etale $Y$-spaces.  By the universal property of $\pi_0(X/Y)$, this functor respects colimits where defined.  Therefore it can be extended to the category of all smooth $X$-spaces, and in particular to $X$, by the following formula:
\begin{equation*}
\pi_0(X/Y) = \varinjlim_{\substack{\text{schemes $X'$} \\ \text{$X' \rightarrow X$ smooth}}} \pi_0(X'/Y) 
\end{equation*}
Note that the colimit is taken in the category of \'etale $Y$-spaces, which is equivalent to the category of \'etale sheaves on the small \'etale site of $Y$ \cite[Theorem~V.1.5]{Milne}.  Since colimits of \'etale sheaves exist, so does the colimit defining $\pi_0(X/Y)$, and it is automatically
\'etale over $Y$.

\change{
\begin{corollary} \label{cor:pi_0-univ-prop}
The conclusion of Proposition~\ref{prop:pi_0-univ-prop} is valid for smooth morphisms from algebraic stacks to schemes.
\end{corollary}
\begin{proof}
Suppose that $X \rightarrow Y' \rightarrow Y$ is a factorization of $X \rightarrow Y$ through an \'etale $Y$-scheme $Y'$.  Then for each smooth $X'$ over $X$, we obtain a factorization $X' \rightarrow Y' \rightarrow Y$ of the map $X' \rightarrow Y$.  By the universal property of $\pi_0(X'/Y)$, this factors uniquely as
\begin{equation*}
X' \rightarrow \pi_0(X'/Y) \rightarrow Y' \rightarrow Y
\end{equation*}
The universal property of the colimit used to define $\pi_0(X/Y)$ now gives the required map $\pi_0(X/Y) \rightarrow Y'$.
\end{proof}
}

\begin{proposition}
Let $X \rightarrow Y$ be a smooth morphism \change{of an algebraic stack to a scheme.}  The formation of the \'etale $Y$-space $\pi_0(X/Y)$ commutes with base change in $Y$.
\end{proposition}
\begin{proof}
\change{Let $Y' \rightarrow Y$ be a morphism of schemes and let $X'$ be the base change of $X$.}
Choose a presentation of $X$ as a colimit of smooth $X$-schemes $X_i$.  Let $X'_i = X_i \fpr_Y Y'$.  Then $X' = X \fpr_Y Y'$ is the colimit of the smooth $X'$-schemes $X'_i$ and so
\begin{equation*}
\pi_0(X'/Y') = \colim \pi_0(X'_i/Y') = \colim (\pi_0(X_i/Y) \fpr_Y Y') = ( \colim \pi_0(X_i/Y) ) \fpr_Y Y'
\end{equation*}
using the commutation of $\pi_0$ with base change for schemes and the fact that colimits of sheaves commute with pullback.
\end{proof}

The proposition allows us to extend the definition of $\pi_0(X/Y)$ to an arbitrary smooth morphism of algebraic stacks:  Let $Y' \rightarrow Y$ be a smooth cover by a scheme and put $X' = X \fpr_Y Y'$.  Then $\pi_0(X'/Y')$ is an \'etale $Y'$-space and this construction is functorial in the $Y$-scheme $Y'$.  Therefore $\pi_0(X'/Y')$ descends to an \'etale $Y$-space $\pi_0(X/Y)$.

\begin{proposition}
The map $X \rightarrow \pi_0(X/Y)$ has connected fibers.
\end{proposition}
\begin{proof}
Since the formation of $\pi_0(X/Y)$ commutes with base change in $Y$, it is sufficient to treat the case where $Y$ is the spectrum of an algebraically closed field.  In that case $\pi_0(X/Y) = \pi_0(X)$ and the assertion is immediate from the definition of~$\pi_0$.
\end{proof}

\subsection{Artin cones}

Suppose that $\sigma$ is a fine, saturated, sharp monoid.  For a logarithmic scheme $(X, M_X)$, define \change{a contravariant functor $\mathcal A_\sigma$ from logarithmic schemes to sets:}
\begin{equation*}
\change{\cA_\sigma(X,M_X)= \Hom\bigl((X,M_X), \cA_\sigma\bigr) :=} \Hom\bigl(\sigma^\vee, \Gamma(X, \oM_X)\bigr),
\end{equation*}
When $\sigma = \bN$ is the monoid of natural numbers, we write $\cA = \cA_\bN$.

\change{We record two key  results of Olsson (recall that, by convention, all logarithmic structures are fine and saturated in this paper):

\begin{proposition} \label{lem:cone-cover}
\begin{enumerate} 
\item 
The functor $\cA_\sigma$ is representable by the logarithmic stack $[V/T]$ where $V$ is the toric variety associated to $\sigma$ and $T$ is its dense torus \cite[Proposition~5.17]{Olsson_log}.
\item   The stacks $\cA_\sigma$, with their natural maps to $\Log$, form a representable \'etale cover \cite[Corollary~5.25 and Remark~5.26]{Olsson}.
\end{enumerate}
\end{proposition}
}

\begin{lemma} \label{lem:stab}
The stabilizer groups of a logarithmic structure over a field is the semidirect product of a finite group and a torus.  In particular, it is affine.
\end{lemma}
\begin{proof}
Let $k$ be a field and let $M$ be a logarithmic structure over $k$.  The automorphisms of $M$ are the semidirect product of the automorphism group of the characteristic monoid $\oM$ and the torus $\Hom(\oM^{\rm gp}, \Gm)$.
\end{proof}

\begin{definition}
A logarithmic algebraic stack isomorphic to $[V/T]$, where $V$ is a toric variety and $T$ is its torus, is called an \emph{Artin cone}.
\end{definition}

\change{%
\subsubsection{Maps of Artin cones}\label{Sec:mapsofcones} Observe that we have $\Gamma(\cA_\sigma, \oM\!_{\cA_\sigma}) = \sigma^\vee$, so that
\begin{equation*}
\Hom(\cA_\sigma, \cA_\tau) = \Hom(\sigma, \tau)
\end{equation*}
for any fine, saturated, sharp monoids $\sigma$ and $\tau$.  In particular, $\Hom(\cA, \cA_\sigma) = \sigma$ and $\Hom(\cA_\sigma, \cA) = \sigma^\vee$.
}

%
%

\begin{definition}\label{Def:atomic}
We call a coherent logarithmic scheme $X$ \emph{atomic} if, when $X$ is stratified by the isomorphism type of the stalks of its characteristic monoid, it has a unique stratum that is closed and connected, and the restriction of the characteristic monoid to this stratum is a constant sheaf.
\end{definition}

\begin{lemma} \label{lem:atomic}
\begin{enumerate}[label=(\roman{*})]
\item Every coherent logarithmic scheme whose strata are locally connected (in particular, every logarithmically smooth logarithmic scheme) has an \'etale cover by atomic logarithmic schemes.
\item If $X$ is an atomic logarithmic scheme and $x$ is a point of the closed stratum of $X$ then $\Gamma(X, \oM_X) \rightarrow \oM_{X,x}$ is an isomorphism.
\end{enumerate}
\end{lemma}
\begin{proof}
For every geometric point $x$ of $X$, there is an \'etale neighborhood $U$ of $x$ such that $M_X$ has a chart by $\oM_{X,x}$, lifting the identity map on $\oM_{X,x}$.  The first assertion is local on $X$, so we can assume this is in fact a global chart.  Removing all components of the closed strata other than the one containing $x$, we can assume that the closed stratum is connected.  The global chart then gives a trivialization of $\oM_X$ over the closed stratum.

For the second claim, observe simply that to give a section of $\oM_X$, we must give a section over each stratum in a way that is compatible with generization.  But every stratum is a generization of the closed stratum, so every section of $\oM_X$ over the closed stratum extends to a global section.  But $\oM_X$ has no monodromy on the closed stratum, so the sections of $\oM_X$ on the closed stratum are the same as the sections at any point.
\end{proof}

\begin{proposition}
An \'etale sheaf on $\cA_\sigma$ is constructible with respect to the stratification of $\cA_\sigma$ associated to its finitely many points\change{, and is constant on each stratum}.
\end{proposition}
\begin{proof} \change{Write $\cA_\sigma = [V/T]$, and stratify $V = \coprod V_i$, with orbits $V_i \simeq T/T_i$. Then $\cA_\sigma$ is stratified as $\cA_\sigma = \coprod [V_i/T] = \coprod \mathrm BT_i$. Since we are working over an algebraically closed field $k$ we have $T_i \simeq \Gm^r$, so
each stratum of $\cA_\sigma$ is isomorphic to $\BGm^r$ for some $r$. It thus suffices to show that all \'etale covers of $\BGm^r$ split.  Indeed, an \'etale cover of  $\BGm^r$ corresponds to a $\Gm^r$-equivariant \'etale cover $\cW \to \Spec k$. This is a scheme of the form $\cW=\coprod_W \Spec k$ with $W$ a finite set. Since  $\Gm^r$ is connected it acts trivially on $W$, hence any $w \in W$ provides an equivariant splitting $\Spec k \to \cW$.}
\end{proof}

\begin{corollary}
Let $F_\sigma$ be the set of faces of $\sigma$, partially ordered by inclusion.  Then the category of \'etale sheaves on $\cA_\sigma$ may be identified with the category of presheaves on $F_\sigma$.
\end{corollary}
\begin{proof}
We may identify the elements of $F_\sigma$ with the strata of $\cA_\sigma$.  Under this identification, inclusion of faces corresponds to specialization.
We now apply the standard description of sheaves that are constructible with respect to a fixed stratification \change{(see \cite[Th\'eor\`eme~9.5.4]{sga4-IV} for the case of two strata; the general case is an immediate induction)}.
\end{proof}

\begin{corollary} \label{cor:sect-over-Artin-cone}
Let $z$ be the unique closed point of $\cA_\sigma$.  If $\cY \rightarrow \cA_\sigma$ is \'etale and representable then the restriction map
\begin{equation*}
\Gamma(\cA_\sigma, \cY) \rightarrow \Gamma(z, \cY)
\end{equation*}
is a bijection.
\end{corollary}
\begin{proof}
Under the identification from the last corollary, $\cA_\sigma$ itself corresponds to the presheaf with constant value a singleton.  Therefore $\Gamma(\cA_\sigma, \cY)$ is determined by its value on the initial object of the category $F_\sigma$, which corresponds to the closed stratum of $\cA_\sigma$.
\end{proof}

\subsection{Artin fans}

\change{
\begin{lemma} \label{lem:fan-cover}
If $\cX$ is an algebraic stack that is representable and \'etale over $\Log$ then $\cX$ has a strict \'etale cover by Artin cones.
\end{lemma}
\begin{proof}
Let $x$ be a point of $\cX$.  By  Proposition~\ref{lem:cone-cover} there exists an Artin cone $\cA_\sigma$ and a map $\cA_\sigma \rightarrow \Log$ that takes the closed point $z$ of $\cA_\sigma$ to the image of $x$.  Let $\cY$ be the base change of $\cX \rightarrow \Log$ to $\cA_\sigma$.  Then $\cY \rightarrow \cA_\sigma$ is \'etale and representable and contains a point $y$ of $\cY$ lying over $x$ in $\cX$ and over $z$ in $\cA_\sigma$.  By Corollary~\ref{cor:sect-over-Artin-cone}, there is a section of $\cY$ over $\cA_\sigma$ passing through $y$.  The image of this section in $\cX$ is an \'etale map whose image contains $x$.  Therefore $\cX$ has a strict \'etale cover by Artin cones.
\end{proof}
}

\begin{definition} \label{def:Artin-fan}
\begin{enumerate}
\item 
We will say that a logarithmic algebraic stack is an \emph{Artin fan} if it has a strict, representable, \'etale cover by Artin cones.
\item 
Let $X$ be a logarithmically smooth logarithmic scheme.  Then the tautological map $X \rightarrow \Log$ is smooth.  Let $\cX = \pi_0(X/\Log)$.  We call $\cX$ the \emph{Artin fan} of~$X$.%
\end{enumerate}
\end{definition}

\begin{remark}
Since Artin cones are \'etale over $\Log$, so are Artin fans.  As Artin cones are representable by algebraic spaces over $\Log$, every strict \'etale morphism from an Artin cone to an Artin fan is representable by algbraic spaces.  In fact, it follows from this observation and Proposition~\ref{prop:factorization}, below, that all morphisms from Artin cones to Artin fans are representable by algebraic spaces.
\end{remark}

\begin{remark}
The Artin fan of a logarithmic scheme $\cX$ is representable over $\Log$ so by Lemma~\ref{lem:fan-cover}, it is an Artin fan in the sense of the first part of the definition.
\end{remark}

\begin{remark}
In \cite{logbd} {we  give} a more general construction of Artin fans for logarithmic schemes that are not necessarily logarithmically smooth.  While these satisfy the same universal property as the Artin fans introduced here, the Artin fan of a general logarithmic scheme $X$ cannot be interpreted in general as $\pi_0(X/\Log)$.  In fact, the morphism from $X$ to its Artin fan need not even be surjective.
\end{remark}

\change{
\begin{lemma} \label{lem:open-point}
Let $\cX$ be a connected Artin fan.  Then $\cX$ has a unique open point, up to isomorphism.
\end{lemma}
\begin{proof}
As $\cX$ is connected, any two open points can be linked by a chain of generizations and specializations.  If $u$ and $v$ are open points of $\mathcal X$ with a common specialization $z$, then  the images of $u$ and $v$ in $\Log$ would be open points with a common specialization.  But $\cX \rightarrow \Log$ is \'etale, and $\Log$ has a unique open point, so $u$ and $v$ have the same image in $\Log$.  Generizations lift uniquely under \'etale maps, so this implies $u = v$.
\end{proof}
}
\change{The Artin fan of an atomic (Definition \ref{Def:atomic}) logarithmically smooth scheme is an Artin cone:}
\change{
\begin{lemma} \label{lem:atomic-fan}
Let $X$ be an \emph{atomic} logarithmically smooth logarithmic scheme and let $\cX$ be its Artin fan.  The natural map $\cX \rightarrow \cA_{\Gamma(X, \oM)^\vee}$ is an isomorphism.
\end{lemma}
\begin{proof}
Let $\cY = \cA_{\Gamma(X,\oM)^\vee}$.  The map $X \rightarrow \cY$ factors uniquely through $\cX$ by the universal property of $\cX$.  The fibers of the map $X \rightarrow \cX$ are precisely the connected components of the stratification of $X$ by isomorphism type of the characteristic monoid.  The closed stratum of $X$ is connected by assumption, so $\cX$ has a unique closed point.  This maps to the unique closed point of $\cY$.  Therefore, by Corollary~\ref{cor:sect-over-Artin-cone}, there is a unique section of $\cX$ over $\cY$.  Let $\cX'$ be the image of this section.  This is an open substack of $\cX$ containing the closed point so its preimage in $X$ is an open subscheme containing the closed stratum.  This means the preimage is all of $X$, and as $X \rightarrow \cX$ is surjective, this means the section $\cY \rightarrow \cX$ is surjective, and is therefore an isomorphism.
\end{proof}
}

\begin{lemma}\label{lem:qcpt-qsep}
\change{%
\begin{enumerate}[label=(\roman{*})]
\item An Artin fan is quasicompact if and only if it has finitely many points.  
\item An Artin fan that is representable over $\Log$ is quasiseparated.
\end{enumerate}
}
\end{lemma}
\begin{proof}
\change{It is immediate that an Artin fan with finitely many points is quasicompact.  Conversely, a quasicompact Artin fan has an \'etale cover by finitely many Artin cones, and an Artin cone has only finitely many points.  This proves the first assertion.

Now suppose that $\cX$ is an Artin fan whose canonical projection to $\Log$ is representable.  We wish to show that the diagonal of $\cX$ is quasicompact.  This is a local assertion on $\cX \times \cX$, and $\cX$ has an \'etale cover by Artin cones, so it is sufficient to show that $\cA_\sigma \mathbin\times_{\cX} \cA_\tau \rightarrow \cA_\sigma \times \cA_\tau$ is quasicompact when $\cA_\sigma$ and $\cA_\tau$ are \'etale over $\cX$.  We note that $\cA_\sigma$ is quasiseparated since it can be presented as $[V/T]$ where $V$ is a toric variety and $T$ is its dense torus.  Therefore it suffices to demonstrate that $\cA_\sigma \mathbin\times_{\cX} \cA_\tau$ is quasicompact.

We argue that $\cA_\sigma \mathbin\times_{\Log} \cA_\tau$ has finitely many points.  Indeed, the fiber over a geometric point of $\cA_\tau$, of which there are only finitely many, corresponds to the fiber of $\cA_\sigma$ over a geometric point of $\Log$.  But if $s$ is the spectrum of an algebraically closed field, a morphism $s \rightarrow \Log$ is a logarithmic structure $M_s$ on $s$, and the lifts to $\cA_\sigma$ correspond to homomorphisms $\sigma^\vee \rightarrow \oM_s$ that can be lifted to charts.  These are in bijection with isomorphisms between $\oM_s$ and $\rho^\vee$ for a face $\rho$ of $\sigma$, and there are only finitely many of these.
}
\end{proof}

\begin{corollary}\label{cor:artin-fan-qcqs}
The Artin fan of a quasicompact, logarithmically smooth logarithmic scheme is quasicompact and quasiseparated.
\end{corollary}
\begin{proof}
\change{%
If $\cX$ is the Artin fan of a quasicompact, logarithmically smooth logarithmic scheme $X$ then $X \rightarrow \cX$ is surjective, so $\cX$ is quasicompact.  By definition, $\cX$ is representable over $\Log$, so it is quasiseparated as well.%
}
\end{proof}

\change{
\begin{remark}\label{Rem:no-cover}
It is not true that every algebraic stack that is strict and \'etale over $\Log$ has an \'etale cover by Artin cones.  For example, let $\Gm$ act on $\bA^1$ by $t . x = t^2 x$.  The quotient $\cY$ is \'etale over $\Log$ when given its natural logarithmic struture.  The map $\cY \rightarrow \Log$ factors through $\cX := \cA$, the Artin fan of $\cY$, as a $\mu_2$-gerbe.  This gerbe is nontrivial because it is the stack of square roots of the tautological bundle on $\cX$, and $\Pic(\cX)  = \mathbf Z$.  Therefore it has no section.

On the other hand, if $\cA_\sigma \rightarrow \cY$ were a strict \'etale map whose image contains the closed point then $\cA_\sigma \rightarrow \cX$ would be a strict \'etale cover.  But the only strict \'etale cover of $\cA$ by an Artin cone is the identity, so if $\cY$ had a cover by an Artin cone, the cover would be a section of $\cY \rightarrow \cX$, and we have just seen there is no such section.
\end{remark}
}

\change{%
\begin{proposition} \label{prop:factorization}
Let $\cX$ be an Artin fan and let $\cA_\sigma \rightarrow \cX$ be a morphism of Artin fans.  Then there is an initial example of a strict \'etale map $\cA_\tau \rightarrow \cX$ and a factorization of the morphism $\cA_\sigma\to \cX$ through a morphism $u:\cA_\sigma \to \cA_\tau$.
\end{proposition}
\begin{proof}
Pick a strict, \'etale map $\cA_\tau \rightarrow \cX$ whose image contains the image of the closed point of $\cA_\sigma$.  Pulling back to $\cA_\sigma$, we get an \'etale cover $\cY$ of $\cA_\sigma$, which necessarily has a section by extending any section over the closed point.  We can then replace $\tau$ with the its smallest face that contains $\sigma$.

Now we argue that the factorization is initial.  Suppose that $\cA_\sigma \rightarrow \cA_{\tau'} \rightarrow \cX$ is another factorization.  Then we obtain a commutative square of solid arrows:
\begin{equation*} \xymatrix{
\cA_\sigma \ar[r]^{\gamma} \ar[d]_{{u}} & \cA_{\tau'} \ar[d] \\
\cA_\tau \ar[r] \ar@{-->}[ur]^-\phi & \cX
} \end{equation*}
We seek a dashed arrow completing the diagram.  As the closed point of $\cA_\sigma$ maps to the closed point of $\cA_\tau$ by assumption, the strict map $\cA_{\tau'} \rightarrow \cX$ covers the image of $\cA_\tau$.  Therefore $\varpi:\cA_{\tau'} \fpr_{\cX} \cA_{\tau} \rightarrow \cA_{\tau}$ is an \'etale cover with a given section $\xi=(\gamma,u)$ over $\cA_\sigma$.  

Note that the composite morphism $\cA_{\tau'} \to \cX \to \Log$ is the tautological morphism, which is representable (Lemma \ref{lem:cone-cover}). Hence $\cA_{\tau'} \to \cX$ is necessarily representable.

The closed point $p_\sigma \in \cA_\sigma$ maps to the closed point $u(p_\sigma) = p_\tau \in \cA_\tau$. It follows that the point $\xi(p_\sigma)$ lies over $p_\tau$,  hence by  Corollary~\ref{cor:sect-over-Artin-cone} there is a unique section $\varsigma= (\phi,id)$ of $\varpi$ sending $p_\tau \mapsto \xi(p_\sigma)$. Moreover, the pullback $\varsigma\circ u$ of $\varsigma$ to $\cA_\sigma$ is determined by its restriction to $p_\sigma$, and thus coincides with $\xi$. Therefore $\phi$ extends $\gamma$, as needed.
\end{proof}
}

\change{
\begin{corollary} \label{Cor:barycenter}
Let $\cY$ be an Artin fan, $\phi: \cA_\tau \rightarrow \cY$ a strict \'etale map and $b:\cA \to \cA_\tau$ its barycenter. Then $\Aut_{\cY}(\phi) = \Aut_{\cY}(\phi b)$.
\end{corollary}
\begin{proof} 
We have:
\begin{equation*}
\Aut_{\cY}(\phi) = \left\{\vcenter{\xymatrix{& \cA_\tau \ar[d]^\phi \\ \cA_\tau \ar@{-->}[ur] \ar[r]_\phi & \cY}} \right\} = \left\{\vcenter{\xymatrix{\cA \ar[r]^b \ar[d]_b & \cA_\tau \ar[d]^\phi \\ \cA_\tau \ar@{-->}[ur] \ar[r]_\phi & \cY}} \right\} = \left\{\vcenter{\xymatrix{\cA \ar[r]^b \ar[d]_b & \cA_\tau \ar[d]^\phi \\ \cA_\tau \ar[r]_\phi & \cY}}\right\} = \Aut_{\cY}(\phi b)
\end{equation*}
In the first equality, we identify $\Aut_{\cY}(\phi)$ as the set of dashed arrows making the triangle $2$-commute.  For the second equality, we observe that every automorphism of $\tau$ commutes with the inclusion of the barycenter.  The third equality is Proposition~\ref{prop:factorization}.  And the fourth equality is the identification of $\Aut_{\cY}(\phi b)$ with the set of $2$-commutative squares as implied by the diagram.
\end{proof}
}

\change{
\begin{example}
Consider the open substack $\cY$ of $\Log$ parameterizing those logarithmic structures that admit charts by $\mathbf N^2$.  This is the Artin fan of the Whitney umbrella \cite[Section 4.6.2]{logsurv}.  For any pair of natural numbers $(a,b)$ we have a map $\mathbf N^2 \rightarrow \mathbf N$.  Viewing $\mathbf N$ as the set of global sections of the characteristic monoid of $\cA$, this induces a logarithmic structure $M(a,b)$ on $\cA$ with a chart by $\mathbf N^2$.

If $a = b = 0$ then the $\cA \rightarrow \cY$ factors through the open point.  If one of $a$ or $b$ is zero but the other is not then $\cA \rightarrow \cY$ is strict, hence is its own initial factorization.  If $a \neq b$ and neither is zero then we have two factorizations $\cA \rightarrow \cA^2$ corresponding to the two maps $(a,b), (b,a) : \mathbf N^2 \rightarrow \mathbf N$.  There is a unique isomorphism $\gamma : \cA^2 \rightarrow \cA^2$ (induced from the automorphism $\mathbf N^2 \rightarrow \mathbf N^2$ sending $(x,y)$ to $(y,x)$) making the diagram
\begin{equation*} \xymatrix{
\cA \ar[r]^{(a,b)} \ar[d]_{(b,a)} & \cA^2 \ar[d] \\
\cA^2 \ar[r] \ar[ur]^{\gamma} &  \cY\ar@{}[r]|{\displaystyle\subset} & \Log
} \end{equation*}
commute.

The final case is where $a = b \neq 0$.  In this case we have a single map $\cA \xrightarrow{(a,a)} \cA^2$ inducing $M(a,a)$ by pullback.  However, a commutative square of solid lines
\begin{equation*} \xymatrix{
\cA \ar[r]^{(a,a)} \ar[d]_{(a,a)} & \cA^2 \ar[d] \\
\cA^2 \ar[r] \ar@{-->}[ur]^{\gamma} &  \cY\ar@{}[r]|{\displaystyle\subset} & \Log
} \end{equation*}
involves the specification of an isomorphism between $M(a,a)$ and itself, commuting with the map to $\mathbf N$.  There are two choices of this isomorphism, the identity and the map that exchanges the generators.  Each of these is induced from a unique dashed arrow making the diagram commute.
\end{example}
}

\subsection{Proper morphisms of Artin fans}

We note that the formation of $\cA_\sigma$ is functorial in $\sigma$.  In particular, if $\sigma$ is a face of $\tau$ then $\cA_\sigma \rightarrow \cA_\tau$ is an open embedding.  If $\Sigma$ is a fan then we can define $\cA_\Sigma$ by gluing the $\cA_\sigma$ associated to the $\sigma \in \Sigma$ along their common faces.  If $V$ is the toric variety associated to $\Sigma$ then $\cA_\Sigma \simeq [V/T]$ where $T$ is the dense torus of $V$.

We observe that $\Hom(\cA, \cA_\Sigma) = | \Sigma |$, where, by definition, $| \Sigma | = \bigcup_{\sigma \in \Sigma} \sigma$.

If a fan
$\Sigma$ is a subdivision of $\tau$ then we call the map $\cA_\Sigma \rightarrow \cA_\tau$ a subdivision.  If $\cY \rightarrow \cX$ is a morphism of Artin fans such that $\cA_\tau \fpr_{\cX} \cY \rightarrow \cA_\tau$ is a subdivision for every strict logarithmic map $\cA_\tau \rightarrow \cX$ then we call $\cY$ a subdivision of $\cX$.

\begin{theorem} \label{thm:proper-fans}
Let $f : \cY \rightarrow \cX$ be a representable morphism of \change{connected Artin fans}.
The following are equivalent:
\begin{enumerate}[label=(\roman{*})]
\item \label{proper:1} $f$ is proper.
\item \label{proper:2} For any strict logarithmic map $\cA_\sigma \rightarrow \cX$, the map $\cY \fpr_{\cX} \cA_\sigma \rightarrow \cA_\sigma$ is proper.
\item \label{proper:3} For any strict logarithmic map $\cA_\sigma \rightarrow \cX$, the map $\cY \fpr_{\cX} \cA_\sigma \rightarrow \cA_\sigma$ is a subdivision.
\item \label{proper:4} $f$ is a subdivision.
\item \label{proper:5} \change{Every map $\cA \rightarrow \cX$ lifts uniquely along $f$:}
any commutative diagram of logarithmic morphisms
\begin{equation} \label{eqn:7} \vcenter{\xymatrix{
& \cY \ar[d]^f \\
\cA \ar[r] \ar@{-->}[ur] & \cX
}} \end{equation}
admits a unique completion by a dashed arrow.
\end{enumerate}
\end{theorem}
\begin{proof}
The equivalence of \ref{proper:3} and \ref{proper:4} is the definition.  The equivalence of \ref{proper:1} and \ref{proper:2} follows by \'etale descent from the fact that the strict logarithmic maps $\cA_\sigma \rightarrow \cX$ form an \'etale cover.

\change{%
For the equivalence of \ref{proper:2} and \ref{proper:3}, it is sufficient to assume \change{$\cX = \cA_\sigma$}. 
Present $\cA_\sigma$ as $[V/T]$, where $V$ is a toric variety and $T$ is its dense torus.  Pulling back via $V \rightarrow \cA_\sigma$, we obtain a map $W \rightarrow V$ whose properness is equivalent to that of $\cY$ over $\cX$, as well as an action of $T$ on $W$ whose quotient is $\cY$.  But $\cY$ has a unique open point by Lemma~\ref{lem:open-point}, which corresponds to a dense orbit of $T$.  Therefore $\cY \rightarrow \cX$ is proper if and only if $W \rightarrow V$ is a proper morphism of toric varieties with the same dense torus. 
By \cite[\S I.2 Theorem 8]{KKMS} or \cite[Section~2.4, Proposition]{Fulton} this is equivalent to the condition that the fan of $W$ subdivides $\sigma$.
}

We check that \ref{proper:3} {implies} \ref{proper:5}.  As any map $\cA \rightarrow \cX$ factors through some strict $\cA_\sigma \rightarrow \cX$ \change{(see Proposition~\ref{prop:factorization})},  it is sufficient to assume that $\cX = \cA_\sigma$.  \change{We can thus assume  $\cY = \cA_\Sigma$ for some subivision $\Sigma$ of $\sigma$.  Consider a diagram of solid lines,~\eqref{eqn:7}.  The component of $\cY$ containing the image of $U$ is isomorphic to $\cA_\Sigma$ for some subdivision $\Sigma$ of $\sigma$.  A lift completing the upper triangle in~\eqref{eqn:7} amounts to giving a map $\cA \rightarrow \cA_\Sigma$, which as we observed above, is equivalent to giving a point of $|\Sigma|$.  But $|\Sigma| \rightarrow |\sigma|$ is a bijection, by definition of a subdivision, so there is a unique such completion making the  \change{triangle} commute.}

\change{%
Now we show that \ref{proper:5} implies \ref{proper:3}.  Suppose that $\cY \rightarrow \cX$ satisfies the lifting criterion of~\ref{proper:5}.  We assume without loss of generality that $\cX = \cA_\sigma$. 
Let $\cA_\tau \rightarrow \cY$ be a strict map.  We argue that the map $\tau \rightarrow \sigma$ induced from the composition $\cA_\tau \rightarrow \cY \rightarrow \cX = \cA_\sigma$ is an injection.  Let $\tau^\circ$ denote the interior of $\tau$, viewed as the set $N_\sigma \cap \sigma_{\mathbb{R}}$ of lattice points of a rational polyhedral cone $\sigma_{\mathbb{R}}$.  It is sufficient to show that $\tau^\circ \rightarrow \sigma$ is injective.  Suppose that $t, u \in \tau^\circ$ have the same image in $\sigma$.  These correspond to maps $t, u : \cA \rightarrow \cA_\tau$ with the same image in $\cX$.  But $\Hom(\cA, \cY) \rightarrow \Hom(\cA, \cX)$ is a bijection, so $t$ and $u$ induce the same map $\cA \rightarrow \cY$.  In other words, we have a commutative square of solid arrows:
\begin{equation*} \xymatrix{
\cA \ar[r]^t \ar[d]_u & \cA_\tau \ar[d] \\
\cA_\tau \ar[r] \ar@{-->}[ur] & \cY
} \end{equation*}
But $t$ and $u$ are interior points of $\tau$, so $\cA_\tau \rightarrow \cY$ is the minimal strict factorization of either of the compositions
\begin{equation*}
\cA \xrightarrow{t,u} \cA_\tau \rightarrow \cY
\end{equation*}
so that Proposition~\ref{prop:factorization} implies there is a unique automorphism of $\cA_\tau$ over $\cY$ sending $u$ to $t$ (the dashed arrow in the diagram).  By Corollary~\ref{Cor:barycenter},
\begin{equation*}
\Aut_{\cY}(\cA_\tau) = \Aut_{\cY}(b),
\end{equation*}
where $b$ denotes the barycenter of $\tau$ and the corresponding map $\cA \rightarrow \cA_\tau$.  But $\Aut_{\cY}(b)$ injects into $\Aut_{\cX}(b)$, because $\cY \rightarrow \cX$ was assumed to be representable.  Therefore this automorphism is the identity and $t = u$.

We conclude that $\cA_\tau \rightarrow \cY$ is injective.  As it is also strict, it is an open embedding and $\tau \subset \sigma$ is a subcone.  A strict family of maps $\cA_\tau \rightarrow \cY$ therefore corresponds to a family of subcones $\tau \subset \sigma$.  Each lattice point of $\sigma$ must be contained in at least one $\tau$.  By Proposition~\ref{prop:factorization}, there is a minimal such cone, up to inclusion of faces.  It follows, therefore, that $\cY$ is representable by a subdivision of $\sigma$.
}
\end{proof}

\subsection{A substitute for functoriality}\label{Sec:relative-Artin}

By construction the map $X \rightarrow \cX$ is strict if we give $\cX$ the logarithmic structure associated to the map $\cX \rightarrow \Log$.  The universal property characterizing $\cX$ implies that \change{$\cX$, and the map $X \rightarrow \cX$}, are functorial in $X$ with respect to strict morphisms.  \change{The same functoriality does not hold for arbitrary morphisms (see \cite[Section 5.4.1]{logsurv}) and we do not know in what generality  Artin fans should be functorial}.  We demonstrate a weaker sort of functoriality below.

Suppose that $Y \rightarrow X$ is a logarithmically smooth morphism of logarithmically smooth schemes.  Let $\cX$ be the Artin fan of $X$.  Then the morphism of logarithmic algebraic stacks $Y \rightarrow \cX$ corresponds to a \emph{smooth} morphism of algebraic stacks $Y \rightarrow \Log(\cX)$.  Let $\cY = \pi_0(Y/ \Log(\cX))$.  Then by composition we have a map from $\cY$ to $\cX$:
\begin{equation*}
\cY \rightarrow \Log(\cX) \rightarrow \cX 
\end{equation*}

\change{%
We verify below that $\cY$ is an Artin fan.  First, note that there is a map $\Log(\cX) \rightarrow \Log$, defined by taking a scheme $S$ with logarithmic structure $M_S$ and morphism $(S, M_S) \rightarrow \cX$ and forgetting the map.

\begin{lemma}
For any Artin cone $\cA_\sigma$, the map $\Log(\cA_\sigma) \rightarrow \Log$ is strict and \'etale.
\end{lemma}
\begin{proof}
Strictness is immediate from the definitions of the logarithmic structures.  The fiber of $\Log(\cX) \rightarrow \Log$ over $(S, M_S)$ is the space of logarithmic maps $(S, M_S) \rightarrow \cA_\sigma$, which map be identified with $\Hom(\sigma^\vee, \Gamma(S, \oM_S))$.  Since $\oM_S$ is a constructible \'etale sheaf, these maps may be identified with the sections of the espace \'etal\'e of the \'etale sheaf $\underline{\Hom}(\sigma^\vee, \Gamma(S, \oM_S))$ on $S$.
\end{proof}

\begin{corollary} \label{cor:rel-over-cone}
Suppose $Y$ is logarithmically smooth over an Artin cone $\cX$.  Then $\pi_0(Y/\Log(\cX)) \simeq \pi_0(Y/\Log)$.
\end{corollary}
\begin{proof}
Since $\Log(\cX)$ is strict and \'etale over $\Log$, these stacks have the same universal property.
\end{proof}

\begin{corollary}
Let $\cX$ be an Artin fan and $Y \rightarrow \cX$ a logarithmically smooth morphism.  Then $\pi_0(Y/\Log(\cX))$ is an Artin fan.
\end{corollary}
\begin{proof}
Let $\cY = \pi_0(Y/\Log(\cX))$.  We must show that $\cY$ has a strict \'etale cover by Artin cones.  Since $\cX$ has a strict, \'etale cover by Artin cones, and $\pi_0(Y/\Log(\cX))$ commutes with base change in $\cX$, we may assume that $\cX = \cA_\sigma$.  In this case, it follows from the previous corollary.
\end{proof}

\begin{definition}
Let $Y \rightarrow X$ be a logarithmically smooth morphism and let $\cX$ be the Artin fan of $X$.  We call $\pi_0(Y/\Log(\cX))$ the relative Artin fan of $Y$ over $X$.
\end{definition}

We record the following lemma for later reference:

\begin{lemma} \label{lem:artin-fan-props}
Let $f : \cY \rightarrow \cX$ be the morphism of Artin fans associated to a logarithmically smooth morphism of logarithmically smooth logarithmic schemes.  Then $f$ is quasiseparated, locally of finite presentation, and its fibers over field-valued points have affine stabilizers.
\end{lemma}
\begin{proof}
We note that all of these properties are \'etale-local in $\cX$.  We can therefore assume that $\cX$ is an Artin cone and therefore that $\cY$ is the Artin fan of a logarithmically smooth logarithmic scheme.  Then $\cY$ is quasiseparated by Corollary~\ref{cor:artin-fan-qcqs} and it is locally of finite presentation because it has an \'etale cover by Artin cones by Definition~\ref{def:Artin-fan}.  In this case, $\cY$ is representable over $\Log$, so if $k$ is a field, the stabilizer group of $y \in \cY(k)$ is a closed subgroup of the automorphism group of the image of $y$ in $\Log$.  The kernel of the map $\Aut_{\cY}(y) \rightarrow \Aut_{\cX}(f(y))$ is a closed subgroup of this finite group, hence is affine.
\end{proof}
}

\subsection{Modifications}

We will only consider logarithmic modifications of logarithmically smooth logarithmic schemes here, so we are content to define logarithmic modifications only in that generality.  F.\ Kato has given a more general definition in \cite[Definition~3.14]{Kato-log-mod}.%
\footnote{One may extend Definition~\ref{def:log-mod} in a manner analogous to \cite[Definition~3.14]{Kato-log-mod} by including morphisms that, \'etale-locally in the target, are fine and saturated base changes of logarithmic modifications in the sense of Definition~\ref{def:log-mod}.  We will not need this generality. \label{foot:log-mod}}

\begin{definition} \label{def:log-mod}
A \emph{logarithmic modification} of a logarithmically smooth logarithmic scheme $X$ is a proper, representable, birational, logarithmically \'etale morphism $Y \rightarrow X$.
\end{definition}

The main purpose of this section is to prove Corollary~\ref{cor:log-mod-str}, which says that all logarithmic modifications are deduced by base change from modifications of Artin fans.  The proof is by reduction, via logarithmic base change, to the case of a logarithmic modification of a logarithmic scheme whose logarithmic structure is the one associated to a smooth Cartier divisor.  As all logarithmic modifications of such logarithmic schemes are isomorphisms, the result is trivial in this case.

We frequently make use in this section of logarithmic changes of base.  Consistent with our assumption that all logarithmic structures are integral and saturated, these fiber products will all be taken in the category of integral, saturated logarithmic schemes.  In order to emphasize this, we refer to ``fine and saturated logarithmic base change'' in the sequel.  Notably, the underlying scheme of a fine and saturated fiber product of logarithmic schemes need not coincide with the fiber product of the underlying schemes in the diagram, unless at least one of the morphisms in the diagram is strict (or, more generally, saturated).

\begin{proposition} \label{prop:proper-base-change}
Let $p : Y \rightarrow X$ be a proper morphism of logarithmic schemes and $X' \rightarrow X$ an arbitrary morphism of logarithmic schemes.  Then the map $p' : Y' \rightarrow X'$ deduced from $p$ by fine and saturated logarithmic base change is also proper.
\end{proposition}

\change{We thank the referee for the following proof, which is significantly simpler than our original argument.}

\begin{proof}\change{%
Consider the scheme-theoretic base change $\uZ' := \uX'\times_\uX\uY$ and the natural morphism  $\uY' \to \uZ'$. Since properness of morphisms of schemes is stable under base change we have that $\uZ' \to \uX'$ is proper. It is shown in \cite[III Corollary 2.1.6]{Ogus} that $\uY' \to \uZ'$ is finite, hence proper. Therefore $\uY' \to \uX'$ is proper, as required.
}
\end{proof}

\begin{corollary} \label{cor:log-mod-bc}
Suppose that $X' \rightarrow X$ is a morphism of logarithmically smooth logarithmic schemes and $Y \rightarrow X$ is a logarithmic modification.  Let $Y' \rightarrow X'$ be the morphism deduced by fine and saturated base change.  Then $Y'$ is a logarithmic modification of $X'$.
\end{corollary}
\change{%
\begin{proof} The schemes $X,X',Y$ and $Y'$ are logarithmically smooth, hence the loci $U_X,U_{X'},U_Y$ and $U_{Y'}$ where the logarithmic structure is trivial are dense. The morphism $U_Y \to U_X$  is an isomorphism, and $U_{X'}$ maps to $U_X$, hence  the pullback $U_{Y'} \to U_{X'}$  is an isomorphism. It follows that $Y' \to X'$ is birational, and, by Proposition~\ref{prop:proper-base-change}, it is proper. The morphism $Y' \to X'$ is logarithmically \'etale by base change, hence it is a logarithmic modification, as required.
\end{proof}%
}
\begin{proposition}
Suppose that $X$ is logarithmically smooth with the logarithmic structure associated to a smooth Cartier divisor.  Then $X$ has no nontrivial logarithmic modifications.
\end{proposition}
\begin{proof}
Let $Y \rightarrow X$ be a logarithmic modification with $X$ as in the statement of the proposition.  We may work locally in $X$ and therefore we may assume that there is a strict map $X \rightarrow \bA^1$, where $\bA^1$ is given its standard toric logarithmic structure.  By Kato's criterion \cite[(3.5.2)]{Kato}, $Y$ is \'etale over $Z \fpr_{\bA^1} X$ for some toric map $Z \rightarrow \bA^1$ with $\dim Z = 1$.  All such $Z$ are quasifinite over $\bA^1$, so $Y$ must be quasifinite over $X$.  On the other hand $Y \rightarrow X$ is proper and birational.  It is therefore an isomorphism, by Zariski's main theorem.
\end{proof}

\begin{corollary} \label{cor:log-mod-df1}
Suppose that $Y$ is a logarithmic modification of a logarithmic scheme $X$ and $X'$ is logarithmically smooth with logarithmic structure associated to a smooth Cartier divisor in $X$.  For any logarithmic morphism $X' \rightarrow X$ the $X'$-scheme $Y'$ deduced from $Y$ by fine and saturated logarithmic base change is isomorphic to $X'$.
\end{corollary}

\begin{corollary} \label{cor:af-proper}
Let $Y \rightarrow X$ be a logarithmically \'etale modification with $\cY \rightarrow \cX$ the associated morphism of Artin fans.  Then the map $\cY \rightarrow \cX$ is proper.
\end{corollary}
\begin{proof}
Recall that by Theorem~\ref{thm:proper-fans}, a representable morphism of Artin fans $\cY \rightarrow \cX$ is proper if and only if it has the right lifting property with respect to the inclusion of the open point in $\cA$.  For a given instance of diagram~\eqref{eqn:7}, we may verify the existence of a lift after changing base via the given map $\cA \rightarrow \cX$.  Set $\cX' = \cA$ and write $X'$, $Y'$, and $\cY'$ for the logarithmic algebraic spaces and stack obtained by fine and saturated base change from $X$, $Y$, and $\cY$, respectively.  We will conclude by showing that $\cY' \rightarrow \cX'$ is an isomorphism.

\change{As $X \rightarrow \cX$ is strict, the fine and saturated pullback respects underlying schemes: $\uX' = \uX \times_{\ucX} \ucX'$. 
Therefore $X' \rightarrow \cX'$ is smooth and surjective with connected geometric fibers.  In particular, $\cX'$ is the Artin fan of $X'$.  For the same reasons, $\cY'$ is the relative Artin fan for $Y' \rightarrow \cX'$.

We have $Y' 
= Y \mathop\times_{X} X'$, so Corollary~\ref{cor:log-mod-bc} implies that $Y' \rightarrow X'$ is a logarithmic modification.  But $X'$ is smooth over $\cX' = \cA$, so the logarithmic structure on $X'$ is the one associated to a smooth Cartier divisor (namely the pullback of the distinguished divisor of $\cA$).  Therefore, by Corollary~\ref{cor:log-mod-df1}, the map $Y' \rightarrow X'$ is an isomorphism.  Thus the universal property of the relative Artin fan implies $\cY' \rightarrow \cX'$ is an isomorphism as well.}
\end{proof}

\begin{corollary} \label{cor:log-mod-str}
If $Y \rightarrow X$ is a logarithmically \'etale modification then $Y \rightarrow \cY \fpr_{\cX} X$ is an isomorphism.
\end{corollary}
\begin{proof}
First of all, $\cY \fpr_{\cX} X$ is a modification of $X$ so $Y \rightarrow \cY \fpr_{\cX} X$ is a modification.  On the other hand, $Y \rightarrow \cY \fpr_{\cX} X$ is strict and both $Y$ and $\cY \fpr_{\cX} X$ are logarithmically \'etale over $X$, so $Y$ is \'etale over $\cY \fpr_{\cX} X$.  Thus $Y \rightarrow \cY \fpr_{\cX} X$ is \'etale, proper, and birational, hence an isomorphism by Zariski's main theorem.
\end{proof}

\numberwithin{theorem}{section}
\section{Algebraicity}\label{Sec:algebraicity}
\label{Sec:M'}
\label{Sec:mcx+mcy}

\change{ 
Let $S$ be a logarithmic scheme and let $\cX$ be an algebraic stack over $S$, equipped with a logarithmic structure.  We write $\fM(\cX/S)$ for the stack whose $T$-points are commutative diagrams
\begin{equation*} \xymatrix{
C \ar[r] \ar[d] & \cX \ar[d] \\
T \ar[r] & S
} \end{equation*}
in which $C \rightarrow T$ is a family of pre-stable logarithmic curves.  We write $\fM(\cX)$ for $\fM(\cX/S)$ when $S$ is a point
}%
\change{%
and we write $\fM$ for $\fM(\cX/S)$ when both $\cX$ and $S$ are points.
}

\begin{proposition} \label{prop:log-etale-over-point}
If $\cX$ is a quasicompact, quasiseparated Artin fan \change{(for example, if $\cX$ is the Artin fan of a quasicompact, quasiseparated, logarithmically smooth logarithmic scheme, or the relative Artin fan of a logarithmically smooth morphism of logarithmically smooth, quasicompact, quasiseparated logarithmic schemes}), then $\fM(\cX)$ is \change{an  algebraic stack.}
\end{proposition}
\begin{proof}
\change{
In view of Lemma~\ref{lem:artin-fan-props}, this is immediate from \cite[Corollary~1.1.1]{minimal}. 
}
\end{proof}


\begin{proposition} \label{prop:rel-alg-stack}
\change{Let $\cY \rightarrow \cX$ be a logarithmically \'etale $S$-morphism of logarithmic algebraic stacks that are locally of finite presentation over $S$.  Assume that $\fM(\cY/S)$ and $\fM(\cX/S)$ are representable by algebraic stacks with logarithmic structures.  Then $\fM(\cY/S) \rightarrow \fM(\cX/S)$ is logarithmically \'etale.}
\end{proposition}

\begin{proof}
\change{
First we verify that $\fM(\cX/S)$ is locally of finite presentation.  We must show that, for any cofiltered system of affine $S$-schemes $T_i$ with limit $T$, the morphism
\begin{equation*}
\varinjlim \fM(\cX/S)(T_i) \rightarrow \fM(\cX/S)(T)
\end{equation*}
is an equivalence.  An object of $\fM(\cX/S)(T)$ is a logarithmic curve $C$ over $T$ and a $T$-morphism $C \rightarrow \cX$.  Since the stack of logarithmic curves is locally of finite presentation, $C$ is induced from a logarithmic curve $C_i$ over some $T_i$, uniquely up to unique isomorphism and enlargement of $i$.  As $\cX$ is locally of finite presentation, the map $C \rightarrow \cX$ is induced from $C_i \rightarrow \cX$ for some, possibly larger value of $i$, again uniquely up to unique isomorphism and further enlargement of $i$.  This proves that $\fM(\cX/S)$ (and likewise $\fM(\cY/S)$) is locally of finite presentation.
}

To prove that \change{$\fM(\cY/S) \rightarrow \fM(\cX/S)$} is logarithmically \'etale, it remains to verify the infinitesimal lifting property.  Consider a strict infinitesimal extension $T \subset T'$ and a lifting problem:
\begin{equation} \label{eqn:10} \vcenter{ \xymatrix{
T \ar[r] \ar[d] & \fM(\cY/S) \ar[d] \\
T' \ar@{-->}[ur] \ar[r] & \fM(\cX/S)
} } \end{equation}
\change{We must show that this diagram has a unique lift.}

This corresponds to a lifting problem
\begin{equation*} \xymatrix{
C \ar[r] \ar[d] & \cY \ar[d] \\
C' \ar@{-->}[ur] \ar[r] & \cX
} \end{equation*}
in which $C \rightarrow \cY$ and $C' \rightarrow \cX$ are, respectively, the families of maps over $T$ and $T'$ defining classified by the diagram~\eqref{eqn:10}.  As $C \subset C'$ is a strict infinitesimal extension, this diagram has a unique lift.
\end{proof}

\begin{remark}
If $\cY \rightarrow \cX$ has logarithmically quasifinite diagonal then the conclusion of the proposition is valid, even without the assumption that $\fM(\cY/S)$ and $\fM(\cX/S)$ be algebraic, in the sense that $\fM(\cY/S) \rightarrow \fM(\cX/S)$ is representable by \'etale morphisms of algebraic stacks.  As we do not need the additional generality, and the proof is more involved, we have omitted it.
\end{remark}

\change{
\begin{proof}[Proof of Proposition~\ref{Prop:mcx-log-smooth}.] By Proposition~\ref{prop:log-etale-over-point} the stack $\fM(\cX)$ is algebraic. Since $\cX \to \Spec k$ is logarithmically \'etale, Proposition~\ref{prop:rel-alg-stack} implies that the morphism $\fM(\cX) \to \fM$ to the stack of prestable curves is logarithmically \'etale. As $\fM$ is logarithmically smooth $\fM(\cX)$ is logarithmically smooth as well, as needed. \end{proof}
}


 Our arguments require another stack
 \change{$\fM(\cY \rightarrow \cX)$}, the moduli space \change{whose objects over a logarithmic scheme $T$ are  diagrams}
\begin{equation}\label{Eq:M'} \vcenter{
\change{
\xymatrix{
C \ar[r] \ar[d] & \cY \ar[d] \\
\oC \ar[r] \ar[d] & \cX  \\
T  & 
}} 
} \end{equation}
in which $C$ and $\oC$ are pre-stable logarithmic curves; \change{Arrows over a logarithmic morphism $T' \to T$ are given by fiber diagrams}.  We write $\fM'(\cY \rightarrow \cX)$ for the open substack of $\fM(\cY \rightarrow \cX)$ where
\begin{enumerate}
\item \label{cond:M'-1} the automorphism group of~\eqref{Eq:M'} relative to its image $\oC \to \cX$ in $\fM(\cX)$ is finite, and
\item \label{cond:M'-2} the stabilization of the map $C \rightarrow \oC$ is an isomorphism (i.e., it is a contraction of unstable rational components).
\end{enumerate}

\begin{proposition} \label{prop:rel-log-etale}
\change{Let $\cY \rightarrow \cX$ be the morphism of Artin fans associated to a logarithmically smooth morphism of logarithmic schemes.  Then the morphism from $\fM(\cY \rightarrow \cX)$ to $\fM(\cX)$ is representable by algebraic stacks.}
\end{proposition}
\begin{proof}
    Working \change{relative  to $\fM(\cX)$} we may assume that a diagram
\change{\begin{equation*} \xymatrix{
\oC_0 \ar[d] \ar[r] & \cX \\
S
} \end{equation*}}
is given.  Then we are to prove that the stack of all logarithmically commutative diagrams
\begin{equation*} \change{\xymatrix{
C \ar[d] \ar[r]  &\cY \fp_{\cX} \oC_0\ar[r]\ar[d] & \cY \ar[d] \\
\oC \ar[r]\ar[d]& \oC_0 \ar[r] \ar[d] & \cX  \\
T \ar[r] & S}} \end{equation*}
\noindent
\change{where the bottom left square is cartesian,} is algebraic.  We may identify this as the space of pre-stable logarithmic  maps \change{$\fM(\cY \fp_{\cX} \oC_0 / S)$}.  \change{But by Lemma~\ref{lem:artin-fan-props}, the map $\cY \rightarrow \cX$ is locally of finite presentation, quasiseparated, and has affine stabilizers, so the same applies to $\cY \fp_{\cX} \oC_0 \rightarrow \oC_0$ by base change; as $\oC_0$ is a family of pre-stable curves, we conclude that $\cY \fp_{\cX} \oC_0$ is also locally of finite presentation, quasiseparated, and has affine stabilizers.  We may therefore apply \cite[Corollary~1.1.1]{minimal}, from which it follows that $\fM(\cY \fp_{\cX} \oC_0 / S)$ is algebraic.}
\end{proof}


\begin{corollary}\label{Cor:M'-algebraic}
Let $\cY \rightarrow \cX$ be a logarithmic morphism between Artin fans.  Then the stack $\fM'(\cY \rightarrow \cX)$ is algebraic.
\end{corollary}
\begin{proof}
By Proposition~\ref{prop:log-etale-over-point}, we know that $\fM(\cX)$ is algebraic.  By \change{Proposition~\ref{prop:rel-log-etale}}
 we deduce that $\fM(\cY \rightarrow \cX)$ is relatively algebraic over $\fM(\cX)$, hence is algebraic.  But the stability condition defining $\fM'(\cY \rightarrow \cX)$ inside $\fM(\cY \rightarrow \cX)$ is open, so it now follows that $\fM'(\cY \rightarrow \cX)$ is algebraic.
\end{proof}

This gives the algebraicity statement of Proposition \ref{Prop:M'}.
 
\numberwithin{theorem}{section}
\section{The cartesian diagram}\label{Sec:cartesian}

Let $Y \rightarrow X$ be a logarithmic modification of proper logarithmic schemes.  Section~\ref{Sec:toric-stacks} provides us a cartesian diagram of logarithmic stacks
\begin{equation} \label{eqn:4} \vcenter{\xymatrix{
Y  \ar[d]\ar[r] & X \ar[d]   \\
\cY\ar[r]  & \cX
}} \end{equation}
in which 
\begin{enumerate}
\item $\cX$ and $\cY$ are Artin fans,
\item  the vertical arrows are strict, smooth, and surjective, and
\item $X$ and $Y$ are proper logarithmic schemes.
\end{enumerate}

We consider the following diagram
\begin{equation} \label{eqn:1} \vcenter{\xymatrix{
\ocM(Y) \ar[r] \ar[d] & \ocM(X) \ar[d] \\
\fM'(\cY \rightarrow \cX) \ar[r] & \fM(\cX)
}} \end{equation}
with the following definitions:
\begin{enumerate}
\item $\ocM(X)$ and $\ocM(Y)$ are, respectively, the moduli stacks of stable logarithmic maps into $X$ and $Y$,
\item $\fM(\cX)$ is the moduli space of pre-stable logarithmic maps into $\cX$, and
\item $\fM'(\cY \rightarrow \cX)$ is the moduli space of  diagrams (\ref{Eq:M'}) described in Section~\ref{Sec:M'} above, \change{with the relative stability condition described there}.
\end{enumerate}

 The map  $\ocM(X) \to \fM(\cX)$ is defined by composition of $C \to X$ with $X \to \cX$.  The map $\fM'(\cY \rightarrow \cX) \to \fM(\cX)$ is obtained by sending a diagram (\ref{Eq:M'}) to the map $\oC \to  \cX$. The map $\ocM(Y) \to \ocM(X)$ is defined using
\cite[B.6]{AMW}: an object $C \to Y$ of $\ocM(Y)$ induces a stable map $\oC \to  X$ \change{by stabilization of the composition $C \rightarrow Y \rightarrow X$; it comes  along  with a commutative} diagram 
\begin{equation*} \xymatrix{
C \ar[r] \ar[d] & Y \ar[d] \\
\oC \ar[r] & X.
} \end{equation*}
Since $Y = X \times_\cX \cY$ this extends uniquely to
\begin{equation*} \xymatrix{
C \ar[r] \ar[d] & Y \ar[d]\ar[r]&\cY \ar[d] \\
\oC \ar[r] & X\ar[r]& \cX.
} \end{equation*}
This gives us the map \change{$\ocM(Y) \to \fM(\cY \rightarrow \cX)$}.  \change{We argue that it takes values in $\fM'(\cY \to \cX)$.  Suppose that $C \rightarrow Y$ is an object of $\ocM(Y)$, let $H$ be the automorphism group of its image in $\fM(\cY \to \cX)$, and let $H'$ be the subgroup of $H$ fixing $\oC \to \cX$.  Let $G$ be the automorphism group of $C \to Y$ and let $G'$ be the subgroup fixing $\oC \to X$.  By the universal property of fiber products, $G'$ and $H'$ are isomorphic; as $G$ is finite, so is $G'$, and therefore so is $H'$.  Thus $\ocM(Y) \to \fM(\cY \to \cX)$ takes values in $\fM'(\cY\to\cX)$, as claimed.}  

\begin{lemma} \label{lem:strictness}
The maps 
\begin{enumerate}[label=(\roman{*})]
\item $\ocM(X) \rightarrow \fM(\cX)$,
\item $\ocM(Y) \rightarrow \fM'(\cY \rightarrow \cX)$, and
\item $\fM'(\cY \rightarrow \cX) \rightarrow \fM(\cY)$
\end{enumerate}
are strict.
\end{lemma}
\begin{proof}
For the map $\ocM(X) \rightarrow \fM(\cX)$ we consider an $S$-point $f : C \rightarrow X$ of $\ocM(X)$.  Let $g : C \rightarrow \cX$ be the induced $S$-point of $\fM(\cX)$.  The minimality condition defining the logarithmic structures of $\ocM(X)$ and $\fM(\cX)$ depend, respectively, only on the morphisms of logarithmic structures $f^\ast M_X \rightarrow M_C$ and $g^\ast M_{\cX} \rightarrow M_C$ on $C$.  As $X \rightarrow \cX$ is strict, these data coincide.

For the map $\ocM(Y) \rightarrow \fM'(\cY \rightarrow \cX)$, first note that $\ocM(Y) \rightarrow \fM(\cY)$ is strict by the previous paragraph.  It is therefore sufficient to show that $\fM'(\cY \rightarrow \cX) \rightarrow \fM(\cY)$ is strict.  The definition of minimality for $\fM(\cY)$ at an $S$-point $f : C \rightarrow \cY$ depends, as above, on the morphism $f^\ast M_{\cY} \rightarrow M_C$.  The definition of minimality for $\fM'(\cY \rightarrow \cX)$ at an $S$-point
\begin{equation*} \xymatrix{
C \ar[r]^f \ar[d]_{\tau} & \cY \ar[d] \\
\oC \ar[r]^g & \cX
} \end{equation*}
depends on the maps $f^\ast M_{\cY} \rightarrow M_C$ and $g^\ast M_{\cX} \rightarrow M_{\oC}$.  However, the latter of these may be constructed from the former as the composition
\begin{equation*}
g^\ast M_{\cX} \rightarrow \tau_\ast f^\ast M_{\cY} \rightarrow \tau_\ast M_C \simeq M_{\oC}
\end{equation*}
taking into account the isomorphism $\tau_\ast M_C \simeq M_{\oC}$ of \cite[Theorem~B.6]{AMW}.  Therefore the minimality conditions depend on the same data, so they yield the same logarithmic structures.
\end{proof}

\begin{proof}[Proof of Proposition \ref{Prop:cartesian}]
We verify that diagram~\eqref{eqn:1} is logarithmically cartesian.  As its vertical arrows are strict by Lemma~\ref{lem:strictness}, this will imply that the underlying diagram of algebraic stacks is cartesian as well.

Suppose that we are given maps $S \rightarrow \fM'(\cY \rightarrow \cX)$ and $S \rightarrow \ocM(X)$ along with an isomorphism between the induced maps \change{$S \rightarrow \fM(\cX)$.}  These data correspond to a diagram of solid lines
\begin{equation} \label{eqn:3} \vcenter{\xymatrix{
C \ar[d] \ar@/^15pt/[rr] \ar@{-->}[r] & Y_S \ar[r] \ar[d] & \cY_S \ar[d] \\
\oC \ar[r] & X_S \ar[r] & \cX_S
}} \end{equation}
of logarithmic algebraic stacks over $S$.  We obtain a map $C \rightarrow Y_S$ completing the commutative diagram by the universal property of the fiber product.  

\change{
It remains to verify that $C \to Y_S$ is stable. }
\change{Fixing diagram~\eqref{eqn:3}, we let $H$ denote the automorphism group of the induced object of $\fM(\cY\to\cX)$.  We let $H''$ be the automorphism group of the induced object $\oC\to\cX_S$ of $\fM(\cX)$ and we take $H'$ to be the kernel of $H \rightarrow H''$, i.e., the subgroup of $H$ fixing $\oC\to\cX_S$.  We define $G$ to be the automorphism group of the induced object of $\fM(Y\to X)$, we take $G''$ to be the automorphism group of $\oC\to X_S$ as an object of $\fM(X)$, and we take $G'$ to be the kernel fo $G \to G''$.  Then $G''$ is finite, because $\oC\to X_S$ is in $\ocM(X)$ by hypothesis, and $H'$ is finite, because the outer square of~\eqref{eqn:3} is in $\fM'(\cY\to\cX)$, by hypothesis.  But $H' \simeq G'$ by the universal property of the fiber product, so we conclude that $G$ is finite.

Now, $\oC\to X_S$ is stable, it must be the stabilization of $C \to X_S$.  In particular, it is determined functorially from $C\to Y_S$, so that the automorphism group the object of $\fM(Y\to X)$ induced from~\eqref{eqn:3} coincides with the automorphism group of the of $C\to Y_S$ as an object of $\fM(Y)$.  Thus $C\to Y_S$ has finite automorphism group and lies therefore in $\ocM(Y)$.}

 

\end{proof}

\numberwithin{theorem}{subsection}
\section{The universal logarithmic modification}\label{Sec:universal}

\change{Let $\cY \rightarrow \cX$ be the morphism of Artin fans associated to $Y \to X$ constructed in Section~\ref{Sec:relative-Artin}.}
We obtain a correspondence
\begin{equation*} \xymatrix{
& \fM'(\cY \rightarrow \cX) \ar[dl] \ar[dr] \\
\fM(\cY) & & \fM(\cX)
} \end{equation*}
where $\fM'(\cY \rightarrow \cX)$ is the moduli space of minimal logarithmic diagrams~\eqref{Eq:M'} constructed in Section~\ref{Sec:M'}.

\subsection{The arrow  $\fM'(\cY \rightarrow \cX) \to \fM(\cY)$: proof of Proposition~\ref{Prop:M'}}\label{Sec:M'toM}
Algebraicity was shown in Corollary \ref{Cor:M'-algebraic} and strictness was shown in Lemma~\ref{lem:strictness}.  All that is left is to show that $\fM'(\cY \rightarrow \cX)$ is logarithmically \'etale over $\fM(\cY)$.  A logarithmic infinitesimal lifting problem
\begin{equation*} \xymatrix{
S \ar[r] \ar[d] & \fM'(\cY \rightarrow \cX) \ar[d] \\
S' \ar[r] \ar@{-->}[ur] & \fM(\cY)
} \end{equation*}
corresponds to a logarithmic extension problem
\begin{equation*} \xymatrix{
C \ar[r] \ar[d] & C' \ar[r] \ar@{-->}[d] \ar@/_15pt/[dd] & \cY \ar[d] \\
\oC \ar[d] \ar@{-->}[r] \ar@/^15pt/[rr] & \oCp \ar@{-->}[d] \ar@{-->}[r] & \cX \\
S \ar[r] & S' .
} \end{equation*}
Now, $C \rightarrow \oC$ is a contraction of unstable components so we may apply \cite[Appendix B]{AMW} to  obtain $\oCp$ \change{uniquely}.  All that is left is to produce the map $\oCp \rightarrow \cX$ and show it is unique.  This follows from the lemmas below.

\begin{lemma}
$\oCp$ is the pushout of the maps $C \rightarrow \oC$ and $C \rightarrow C'$ in the category of logarithmic schemes.
\end{lemma}

\begin{proof} \change{As $C \rightarrow C'$ and $\oC \rightarrow \oCp$ are homeomorphisms, the underlying topological space of $\oCp$ is the pushout of the maps underlying $C \rightarrow \oC$ and $C \rightarrow C'$.}   
\change{As in \cite[Lemma B.1]{AMW} the structure sheaf of $\oCp$ is the pushforward of the structure sheaf of $C'$, hence $\oCp$ is the pushout of underlying schemes.}
Also the logarithmic structure on $\oCp$ is \change{constructed as} the pushforward of the logarithmic structure on $C'$. This implies the result.
\end{proof}

\begin{lemma}
$\oCp$  is the pushout of the maps $C \rightarrow \oC$ and $C \rightarrow C'$ in the $2$-category of logarithmic stacks.
\end{lemma}
\begin{proof}
The construction of $\oCp$ is local in the \'etale topology of $\oC$, so we may work \'etale locally in $\oC$.  We may therefore assume that given maps $\oC \rightarrow \cX$ and $C' \rightarrow \cX$ factor through a smooth, strict chart $X \rightarrow \cX$.  But then these maps extend uniquely in a compatible way to $\oCp \rightarrow X$ by the previous lemma.  The uniqueness of this extension guarantees that the induced map $\oCp \rightarrow \cX$ is independent of the chart and therefore descends.
\end{proof}

\subsection{Birationality: proof of Proposition \ref{Prop:Costello}}
\begin{proposition} \label{prop:univ}
Suppose that $\cY \rightarrow \cX$ is a logarithmic modification of Artin fans.  Then the maps $\fM'(\cY \rightarrow \cX) \rightarrow \fM(\cX)$ and $\fM'(\cY \rightarrow \cX) \rightarrow \fM(\cY)$ are birational.
\end{proposition}
\begin{proof}
All of the stacks $\fM'(\cY \rightarrow \cX)$, $\fM(\cX)$, and $\fM(\cY)$ are logarithmically smooth.  Therefore they have dense open substacks where their logarithmic structures are trivial.  We show that these dense open substacks are all isomorphic to one another \change{by the given maps}.

Consider an $S$-point of $\fM'(\cY \rightarrow \cX)$, where $S$ has the trivial logarithmic structure.  We have a commutative diagram
\begin{equation} \label{eqn:2} \vcenter{\xymatrix{
C \ar[r] \ar[d] & \cY \ar[d] \\
\oC \ar[r] & \cX 
}} \end{equation}
with both $C$ and $\oC$ logarithmically smooth.  This implies first that the underlying curves of $C$ and $\oC$ are smooth, and second that the map of schemes underlying $C \rightarrow \oC$ is a branched cover.  But the stabilization of $C \rightarrow \oC$ must also be an isomorphism, so its degree must be $1$ and therefore $C \rightarrow \oC$ is an isomorphism.  This proves that $\fM(\cY \rightarrow \cX) \rightarrow \fM(\cY)$ is an isomorphism over the loci with trivial logarithmic structures.

Now consider an $S$-point $\oC \rightarrow \cX$ of $\fM(\cX)$, \change{where $S$ is still assumed to have trivial logarithmic structure, so the underlying curve of $\oC$ is smooth}.  Consider the fine and saturated base change $C = \oC \fp_{\cX} \cY$.  The fibers of $\oC$ over $S$ are logarithmically smooth and the logarithmic structure of each is associated to a smooth Cartier divisor.  Therefore $C = \oC$ by Corollary~\ref{cor:log-mod-df1}.  This immediately yields a section of $\fM'(\cY \rightarrow \cX)$ over \change{the locus where the logarithmic structure is trivial.}  It remains only to verify that if~\eqref{eqn:2} is an $S$-point of $\fM'(\cY \rightarrow \cX)$ then $C = \oC \fp_{\cX} \cY$.  However, this follows from the fact that $C \rightarrow \oC$ is an isomorphism, as we saw above.
\end{proof}

\section{Obstruction theories}

\subsection{The arrow $\ocM(X) \to \fM(\cX)$}\label{Sec:obs X}
First we show that the natural obstruction theory for $\ocM(X)$ over $\fM(\cX)$ agrees with the one over \change{$\Log(\fM)$} defined in~\cite{AC,Chen,GS}.  Let $S \subset S'$ be a strict square-zero extension over $\fM(\cX)$ with ideal $J$ and assume given an $S$-point of $\ocM(X)$.  We have a diagram of solid lines
\change{
\begin{equation} \label{eqn:5} \vcenter{\xymatrix{
& & X \ar[d]  \\
\oC \ar[r] \ar@/^15pt/[urr]^f \ar[d] & \oC' \ar[r] \ar@{-->}[ur] \ar[d] & \cX \\
S \ar[r] & S' .
}} \end{equation}}
Note that because $\cX$ is \'etale over $\Log$, lifts of this diagram are precisely the same as lifts of the diagram
\change{
\begin{equation*} \xymatrix{
& & X \ar[d]  \\
\oC \ar[r] \ar@/^15pt/[urr]^f \ar[d] & \oC' \ar[r] \ar@{-->}[ur] \ar[d] & \Log \\
S \ar[r] & S' .
} \end{equation*}}
Since $X$ is smooth over $\cX$, the logarithmic lifts of either of these diagrams form a torsor on $C$ under the sheaf of abelian groups $f^\ast T_{X / \cX} \tensor J = f^\ast T_X^{\log} \tensor J$.  Therefore if we define $\cE(J)$ to be the stack on $S$ of $f^\ast T_{X}^{\log} \tensor J$-torsors on $C$ we obtain an obstruction theory in the sense of \cite{obs} for $\ocM(X)$ \change{over $\fM(\cX)$} or over $\Log(\fM)$.  The latter of these is the one defined in \cite{AC,Chen,GS}.

\subsection{The arrow $\ocM(Y) \to \fM'(\cY \to\cX)$}
A similar argument will apply to give the obstruction theory for $\ocM(Y)$ \change{over $\fM(\cY)$.}  \change{Since  $\fM'(\cY \to\cX) \to \fM(\cY)$ is \'etale this also serves as an obstruction theory over $\fM'(\cY \to\cX)$. Explicitly,}  a lifting problem
\begin{equation*} \xymatrix{
S \ar[r] \ar[d] & \ocM(Y) \ar[d] \\
S' \ar@{-->}[ur] \ar[r] & \fM'(\cY \rightarrow \cX)
} \end{equation*}
corresponds to the following lifting problem:
\begin{equation} \label{eqn:6} \vcenter{\xymatrix{
& & Y \ar[d] \\
C \ar[r] \ar[d]_\tau \ar@/^15pt/[urr]^g & C' \ar@{-->}[ur] \ar[r] \ar[d] & \cY \ar[d] \\
\oC \ar[r] \ar[d] & \oCp \ar[r] \ar[d] & \cX  \\
S \ar[r] & S' .
}} \end{equation}
As before, the lifts form a torsor under $g^\ast T_{Y / \cY} \tensor J = g^\ast T_Y^{\log} \tensor J$.  Taking $\cE'(J)$ to be the stack on $S$ parameterizing torsors on $C$ under $g^\ast T_Y^{\log} \tensor J$ therefore gives a perfect relative obstruction theory.

 \change{We now demonstrate that the obstruction for $\ocM(X)$ considered in Section~\ref{Sec:obs X} pulls back to \emph{the same} obstruction theory for $\ocM(Y)$ over $\fM'(\cY \rightarrow \cX)$.}  

\change{Recall that $\cE(J)$ was defined in Section~\ref{Sec:obs X}} to be the stack of $f^\ast T_X^{\log} \tensor J$-torsors on $\oC$ and the obstruction was the torsor of lifts of the diagram below:
\begin{equation} \label{eqn:8}\change{ \xymatrix{
\oC \ar[r]^f \ar[d] & X \ar[d] \\
\oCp \ar[r] \ar@{-->}[ur] & \cX}
} \end{equation}

To identify $\cE(J)$ with $\cE'(J)$ we note that because $C \rightarrow \oC$ is a contraction of chains of rational curves \change{and $Y \to X$ logarithmically \'etale,} 
\begin{equation*}
R \tau_\ast (g^\ast T_Y^{\log} \tensor J) = R \tau_\ast (\tau^\ast f^\ast T_X^{\log}  \tensor J) = f^\ast T_X^{\log} \tensor J
\end{equation*}
so that $g^\ast T_Y^{\log} \tensor J$-torsors on $C$ may be identified with $f^\ast T_X^{\log} \tensor J$-torsors on $\oC$ by pullback.  Moreover, the compatibility of the lifting problems
\begin{equation*} \vcenter{\xymatrix{
C \ar[r] \ar[d] & \oC \ar[r] \ar[d] & X \ar[d] \\
C' \ar@{-->}[urr] \ar[r] & \oCp \ar@{-->}[ur] \ar[r] & \cX
}} \end{equation*}
ensures that the torsors of lifts of~\eqref{eqn:6} and~\eqref{eqn:8} are identified.  This shows that the obstruction theories coincide.

\subsection{Conclusion} We have therefore proved the following precise restatement of Proposition~\ref{Prop:relative-obstruction}:
\begin{proposition} \label{prop:obs}
Let $\cE$ denote the perfect relative obstruction theory for $\ocM(X)$ over $\Log(\fM)$ and let $\cE'$ denote the perfect relative obstruction theory for $\ocM(Y)$ over $\Log(\fM)$.  Then 
\begin{enumerate}
\item $\cE$ is also a perfect relative obstruction theory for $\ocM(X)$ over $\fM(\cX)$, \change{and in particular $[\ocM(X)]^\vir = (\psi_X)_{\cE}^! [{\fM(\cX)}]$.}
\item $\cE'$ is also a perfect relative obstruction theory for $\ocM(Y)$ over $\fM'(\cY \rightarrow \cX)$,  \change{and in particular $ [\ocM(Y)]^\vir =   {(\psi'_Y)_{\cE'}^! [\fM'(\cY\to \cX)].}$}
\item \change{$\ocM(h)^\ast \cE = \cE'$.}
\end{enumerate}
\end{proposition}

We may now combine Propositions~\ref{prop:univ} and~\ref{prop:obs} with Costello's theorem~\cite[Theorem~5.0.1]{Costello} to deduce \change{$\ocM(h)_\ast [\ocM(Y)]^\vir = [\ocM(X)]^\vir$.}

\bibliographystyle{amsalpha}             
\bibliography{loginv}       
\end{document}